\documentclass{amsart}
\usepackage{graphicx,psfrag}

\theoremstyle{plain}
  \newtheorem{thm}{Theorem}
  \newtheorem{cor}[thm]{Corollary}

  \newtheorem{lemma}[thm]{Lemma}
  \newtheorem{prop}[thm]{Proposition}

\theoremstyle{definition}
  \newtheorem{defn}[thm]{Definition}
  \newtheorem{rmk}[thm]{Remark}
  \newtheorem{ques}[thm]{Question}

\newcommand{\RTR}{RT\!\reflectbox{R}\ }

\newcommand{\C}{\mathbb C}

\newcommand{\R}{\mathbb R}

\newcommand{\bd}{\partial}

\begin{document}

\title{Tangle sums and factorization of A-polynomials}

\author[Masaharu Ishikawa]{Masaharu Ishikawa$^1$}

\author[Thomas W.\ Mattman]{Thomas W.\ Mattman$^2$}

\author[Koya Shimokawa]{Koya Shimokawa$^3$}

\begin{abstract}
We show that there exist infinitely many examples of pairs of knots, $K_1$ and $K_2$,
that have no epimorphism $\pi_1(S^3\setminus K_1)\to \pi_1(S^3\setminus K_2)$
preserving peripheral structure although their A-polynomials have the factorization
$A_{K_2}(L,M)\mid A_{K_1}(L,M)$.
Our construction accounts for most of the known factorizations of this form 
for knots with $10$ or fewer crossings.
In particular, we conclude that 
while an epimorphism will lead to a factorization of A-polynomials, 
the converse generally fails.
\end{abstract}

\dedicatory{$^1$Mathematical Institute, Tohoku University\\
Sendai 980-8578, Japan\\
ishikawa@math.tohoku.ac.jp\\
\medskip
$^2$Department of Mathematics and Statistics,  California State
University, Chico,\\
Chico CA 95929-0525, USA\\
TMattman@CSUChico.edu\\
\medskip
$^3$Department of Mathematics,  Saitama University\\  
Saitama 338-8570, Japan\\
kshimoka@rimath.saitama-u.ac.jp
}

\maketitle

\section{Introduction}

Cooper et al.\ \cite{CCGLS} introduced the A-polynomial as a knot  invariant
derived from the $SL(2,\C)$-representations of the fundamental group
of the knot's complement. It is a polynomial in the variables $M$ and $L$, which correspond
to the eigenvalues of the $SL(2,\C)$-representations of the meridian
and longitude respectively.
We can obtain a lot of geometric information from A-polynomials
including boundary slopes of incompressible surfaces in the knot
complement and the non-existence of Dehn surgeries yielding $3$-manifolds
with cyclic or finite fundamental groups,
see for instance~\cite{CGLS, CCGLS, BZ}.

It is natural to ask if there is a correspondence between epimorphisms among
the fundamental groups of knot complements and their A-polynomials.
Actually, Silver and Whitten~\cite{SW} showed that if there exists
an epimorphism, $\pi_1(S^3\setminus K_1)\to \pi_1(S^3\setminus K_2)$, 
between the fundamental groups of two knot complements,
that preserves peripheral structure,
then the A-polynomial of $K_1$ has a factor corresponding to
the A-polynomial of $K_2$ under a suitable change of coordinates.
Here we say an epimorphism preserves peripheral structure
if the image of the subgroup generated by the meridian and longitude of $K_1$
is included in the subgroup generated by the meridian and longitude of $K_2$.
Hoste and Shanahan~\cite{HS} refined this by demonstrating that
the A-polynomial of $K_1$ has a factor which corresponds to 
the A-polynomial of $K_2$ under the change of coordinates $(L,M)\mapsto (L^d,M)$
for some $d\in\mathbb Z$. Ohtsuki, Riley, and Sakuma~\cite{ORS}
made a systematic study of epimorphisms between $2$-bridge link groups.

In this paper, we study factorizations of A-polynomials of knots obtained
by specific tangle sums and the existence of epimorphisms.
In particular, we show that there are infinitely many knots whose A-polynomials
have factorizations for which there is no corresponding epimorphism.
Moreover, our factorization is realized without change of coordinates.
We found 16 examples of such factorizations of A-polynomials
among the knots with 10 or fewer crossings. 

We now introduce the tangle sum, which will play a central role in this paper.
A marked tangle is one whose four ends have specific orientations as shown 
on the left in Figure~\ref{fig1}.
The sum of two marked tangles $S$ and $T$ is a marked tangle obtained as shown on the right,
denoted by $S\text{\.+}T$.
Let $N(T)$ and $D(T)$ denote the numerator
and denominator closure of a marked tangle $T$ respectively.

\begin{figure}[htbp]
   \centerline{\input{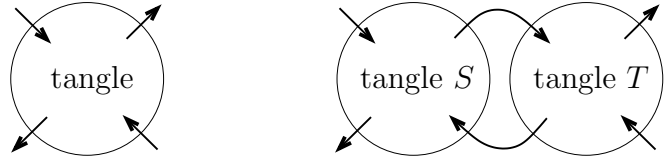}}
   \caption{A marked tangle and the sum of marked tangles.\label{fig1}}
\end{figure}

First we consider the factorization of the Alexander polynomial of a knot $N(S\text{\.+}T)$.
Let $\Delta_K(t)$ denote the Alexander polynomial of a knot $K$ in $S^3$.
Using his formulation of the Alexander polynomial, Conway observed that
\[
    \Delta_{N(S\text{\.+}T)}(t)=\Delta_{N(T)}(t)\Delta_{D(S)}(t)+\Delta_{D(T)}(t)\Delta_{N(S)}(t)
\]
holds (cf.~\cite[Theorem~7.9.1]{C}).
In particular, if $N(S)$ is a split link then the Alexander polynomial has a factorization as
\begin{equation}\label{alexfac}
    \Delta_{N(S\text{\.+}T)}(t)=\Delta_{N(T)}(t)\Delta_{D(S)}(t)
\end{equation}
since $\Delta_{N(S)}(t)=0$.

If, for knots $K_1$ and $K_2$, 
$\pi_1(S^3\setminus K_1)$ has an epimorphism onto $\pi_1(S^3\setminus K_2)$, then 
$\Delta_{K_2}(t) \mid \Delta_{K_1}(t)$ (e.g., see \cite{CF}). It is known that
converse does not hold in general.
If we restrict our attention to epimorphisms which preserve peripheral
structure, we can find an infinite family of counterexamples to the converse
in $2$-bridge knots, which is our first result.

\begin{thm}\label{thm01}
Let $K=K(\beta/\alpha)$ be a $2$-bridge knot, where $\alpha/\beta$ has continued fraction $[2,-n,k,n,-2]$, and $K_{2,k}$ be the $(2,k)$-torus knot,
where $k>2$ is odd and $n>1$. Then $\pi_1(S^3\setminus K)$ admits no epimorphism 
onto $\pi_1(S^3\setminus K_{2,k})$ preserving peripheral structure, 
although $\Delta_{K_{2,k}}(t)\mid \Delta_{K}(t)$.
\end{thm}

Note that the $K(\beta/\alpha)$ has a diagram of the form $N(S\text{\.+}T)$
with $N(T)=K_{2,k}$, see~Figure~\ref{fig2}.
This is why we have the factorization $\Delta_{K_{2,k}}(t)\mid \Delta_{K}(t)$.

\begin{figure}[htbp]
   \centerline{\input{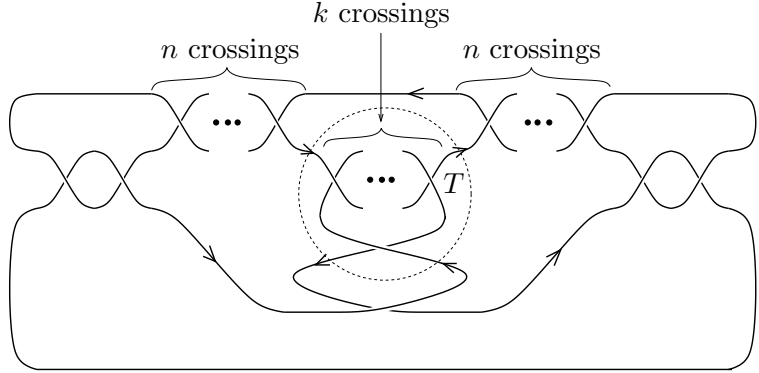}}
\caption{The $2$-bridge knot $K(\beta/\alpha)$ where $\alpha/\beta=[2,-n,k,n,-2]$.\label{fig2}}
\end{figure}

As in Figure~\ref{fig4} below (in Section 4), we can easily see that even if the tangle $T$ is not marked,
by applying Reidemeister move II, we can represent $K$ as the sum of two marked tangles $S'$ and $T'$.
Moreover we have $N(T)=N(T')$ and $N(S)=N(S')$ (but usually $D(T) \ne D(T')$ and $D(S) \ne D(S')$). 
Therefore we always have $\Delta_{N(T)}(t)\mid \Delta_K(t)$ without assuming that the tangles are marked.
We denote by $S+T$ the sum of non-marked tangles $S$ and $T$.
As in the marked tangle case, $N(T)$ and $D(T)$ denote the numerator and 
denominator closure, respectively, of a non-marked tangle $T$.

Next, we study the factorization of the A-polynomial of the knot $N(S+T)$.
Let $A_K(L,M)$ denote the A-polynomial of a knot $K$ in $S^3$
and $A^\circ_K(L,M)$ denote the product of the factors of $A_K(L,M)$ containing the variable $L$.

\begin{thm}\label{thm02}
Suppose that $N(S+T)$ and $N(T)$ are knots and $N(S)$ is a split link in $S^3$. 
Then $A^\circ_{N(T)}(L,M)\mid A_{N(S+T)}(L,M)$.
\end{thm}

Note that $A_K(L,M)/A^\circ_K(L,M)$ is a polynomial with only one variable $M$.
It is known that the roots of such a polynomial are on the unit circle, for instance see~\cite{CCGLS}.
An example is the knot $9_{38}$ which, according to a calculation by Culler \cite{KnotInfo},
has $(1-M^2)^2$ as a factor.

Several properties arising from the $SL(2,\C)$-representations of the fundamental group
of the complement of $N(T)$ are inherited by $N(S+T)$ as listed in the following corollary.
The third statement of the corollary refers to $r$-curves, which we now define.

\begin{defn}
If $A_K(L,M)$ has a factor of the form 
$1 \pm L^bM^a$
(respectively, 
$L^b \pm M^a$ 
), we say
that the character variety of the knot $K$ has an $r$-curve with $r = a/b$ (respectively, $r = -a/b$).
\end{defn}

\begin{rmk} This corresponds to the definition of  \cite[Section 5]{BZ2}.
 \end{rmk}

\begin{cor}\label{cor03}
Suppose that $N(S+T)$, $N(T)$, and $N(S)$ satisfy the conditions of Theorem~\ref{thm02}.
\begin{itemize}
\item[(1)] 
With the possible exception of $\frac10$, 
the set of boundary slopes of $N(T)$ detected by its character variety 
is a subset of the boundary slopes of $N(S+T)$.
\item[(2)]
If $A^\circ_{N(T)}(L,M)$ has a Newton polygon 
which does not admit cyclic/finite surgeries, 
then $N(S+T)$ does not have cyclic/finite surgeries.
\item[(3)]
If the character variety of $N(T)$ has an $r$-curve 
with $r \neq \frac10$,
then that of $N(S+T)$ also has an $r$-curve with the same $r$. 
\end{itemize}
\end{cor}

Our second corollary shows that the pairs of knots of Theorem~\ref{thm01} also constitute an infinite family where the $A$-polynomial of one factors 
that of the other even though there is no epimorphism between them.

\begin{cor}\label{cor04}
Let $K$ be the $2$-bridge knot $K(\beta/\alpha)$ with 
$\alpha/\beta=[2,-n,k,n,-2]$ and $K_{2,k}$ the $(2,k)$-torus knot,
where $k>2$ is odd and $n>1$. Then 
$\pi_1(S^3\setminus K)$ admits no epimorphism onto $\pi_1(S^3\setminus K_{2,k})$ 
preserving peripheral structure, although $A_{K_{2,k}}(L,M)\mid A_{K}(L,M)$.
\end{cor}

This corollary follows from Theorem~\ref{thm01}, Theorem~\ref{thm02},
and that $\infty$ is not a boundary slope of $K_{2,k}$. 

In \cite{Riley} Riley discusses three ways in which character varieties of
$2$-bridge knots and links may become reducible.
The examples in Corollary \ref{cor04} do not fall into any of those
three categories.

This paper is organized as follows.
We prove Theorems~\ref{thm01} and \ref{thm02}
in Sections~2 and 3 respectively.
The 16 examples of factorizations of A-polynomials are
listed in Table~\ref{table1} of Section~4 where we also pose
a few questions. In an appendix we explain
the factorization of the Alexander polynomials
$\Delta_{N(T)}(t) \mid \Delta_{N(S+T)}(t)$ from the viewpoint
of $SL(2,\C)$-representations.

We would like to express our gratitude to Makoto Sakuma for his
precious comments and for informing us of Riley's result on the reducibility of
the character variety of $2$-bridge knots.
We would like to thank Fumikazu Nagasato for telling us some important details
about A-polynomials as relates to Theorem~\ref{thm02}.
The first author is supported by MEXT, Grant-in-Aid for Young Scientists (B)
(No. 22740032).
The third author is supported by the MEXT, Grant-in-Aid for Scientific Research (C) (No. 22540066).

In this study, we often referred to the list of A-polynomials computed by Hoste and Culler
and other knot invariants in the database {\tt KnotInfo}~\cite{KnotInfo}.
We also used the program {\tt Knotscape} of Hoste and Thistlethwaite
for checking the knot types of given knot diagrams.
We thank them for these useful computer programs and their database.

\section{Proof of Theorem~\ref{thm01}}

\noindent
{\it Proof of Theorem~\ref{thm01}.\;\,}
We follow the notation of \cite{GAR} so that 
$${[} 2, -n, k, n, -2 {]} = \frac{k(2n-1)^2}{kn(2n-1)+1}.$$
Then,  $\alpha = k(2n-1)^2$ and $\beta = kn(2n-1)+1$,  both are positive odd integers, and $0<\beta<\alpha$.

In \cite[Theorem 16]{GAR2}, the authors observe that a knot $K$ 
admits an epimorphism (preserving peripheral structure) onto the group of a torus knot
if and only if $K$ has property Q. In \cite{GAR}, they
present an algorithm that will determine whether
or not a given fraction for a $2$--bridge knot will result in a knot 
with Property Q. We will use that algorithm to show that $K$ 
does not have Property Q.

We must argue that 
$\mbox{gcd}(\alpha, \beta) = 1$. To this end, 
note that we can write 2 as a combination of $\alpha$ and $\beta$:
$$2 = (kn-1)\alpha + (2(1-kn)+k)\beta.$$ Since $\alpha$ is odd, we
can write $\alpha = 1 + 2m$. Then
\[
\begin{split}
1 & = \alpha - 2m \\
& = \alpha -m ((kn-1)\alpha + (2(1-kn)+k)\beta) \\
& = (1-m(kn-1)) \alpha -m(2(1-kn)+k) \beta
\end{split}
\]
whence $\mbox{gcd}(\alpha, \beta) = 1$, as required.

Following the algorithm of \cite[Remark~(2) on p.452]{GAR},
in Step~0, we set $d = \alpha$, $q = \alpha/\beta$.
Since $\mbox{gcd}(\alpha, \beta) = 1$, 
$q$ is an integer only if $\beta =1$. However, $\beta = kn(2n-1)+1 > 1$. So we pass on to step~2.
Notice that $\alpha < 2 \beta$. Then $\lfloor q \rfloor = 1$, where 
$\lfloor x\rfloor=\max\{y\in\mathbb Z\mid y\leq x\}$ for $x\in\R$. So, 
$\mbox{gcd}(\lfloor q \rfloor, d) = 1$. On the other hand, since $d = \alpha$ is
odd, $\mbox{gcd}(\lfloor q \rfloor+1, d) = 1$, too. Thus 
$d':=\max\{\gcd(\lfloor q\rfloor, d), \gcd(\lfloor q\rfloor +1,d)\}=1$ and 
$K$ does not have property Q, as we wished to show.

As mentioned in the introduction, since $K = N(S\text{\.+}T)$ with $N(S)$ a trivial link of two components
and $N(T) = K_{2,k}$, then $\Delta_{K_{2,k}}(t)\mid \Delta_{K}(t)$.
\qed

\section{Proof of Theorem~\ref{thm02}}

We prove Theorem~\ref{thm02} in this section.
Let $F_2$ denote the free group of rank $2$.
We first introduce a lemma that allows us a specific choice for the generators of $F_2$.

\begin{lemma}\label{lemma100}
Let $\langle a, b\rangle$ be generators of $F_2$ and $\hat a$ be an element in $F_2$
conjugate to $a$. Then there exists $\hat b\in F_2$ conjugate to $b$
such that $\hat a$ and $\hat b$ generate $\langle a, b\rangle =F_2$.
\end{lemma}

\begin{proof}
By hypothesis, there exists $c\in F_2$ such that $\hat a=cac^{-1}$.
Set $\hat b=cbc^{-1}$. Let $\varphi:F_2\to \langle \hat a, \hat b\rangle \subset \langle a, b\rangle=F_2$ 
be a homomorphism defined by $\varphi(x)=cxc^{-1}$.
If $\varphi(x_1)=\varphi(x_2)$ then $x_1=x_2$, hence $\varphi$ is injective.
Since $\varphi(c^{-1}yc)=c(c^{-1}yc)c^{-1}=y$ for any $y\in \langle a, b\rangle$,
$\varphi$ is a map onto $\langle a, b\rangle=F_2$. 
Therefore $\varphi:F_2\to F_2$ is an isomorphism and $\langle \hat a, \hat b\rangle$ generate $F_2$.
\end{proof}

Let $N(S+T)$, $N(S)$, and $N(T)$ be as in Theorem~\ref{thm02}.
Since $N(S+T)$ is a knot, the split link $N(S)$ consists of two link components, say $S_1$ and $S_2$.
Since $\pi_1(S^3\setminus N(S))\cong \pi_1(S^3\setminus S_1)*\pi_1(S^3\setminus S_2)$, 
the abelianizations $\pi_1(S^3\setminus S_i)\to H_1(S^3\setminus S_i)\cong \mathbb Z$,
$i=1,2$, define a quotient map $q:\pi_1(S^3\setminus N(S))\to F_2$ 
that sends meridians of the two different components to the two generators $a$ and $b$ of $F_2$.
Set $\hat a, b'$ to be the elements in $F_2=\langle a, b\rangle$ 
corresponding to the meridional loops around the two strands of the numerator closure of the tangle $T$.
By replacing $a$ (resp.\ $b$) by its inverse element if necessary,
we may assume that $a$ and $\hat a$ (resp.\ $b$ and $b'$) are conjugate.
By Lemma~\ref{lemma100}, there exists an element $\hat b$ conjugate to $b$ 
such that $\hat a$ and $\hat b$ generate $F_2=\langle a, b\rangle$.
Since $b'$ is conjugate to $b$, there exists $c\in \langle \hat a, \hat b\rangle$ such that
$b'=c\hat bc^{-1}$.
We further assume that the elements in $\pi_1(S^3\setminus N(T))$ 
corresponding to $\hat a$ and $b'$ are conjugate
by replacing one of them by its inverse element if necessary.

Let $\rho_0$ be a representation in 
$\text{\rm Hom}(\pi_1(S^3\setminus N(T)),SL(2,\C))$.

\begin{lemma}\label{lemma101z}
Suppose that $\rho_0(\hat a)=\begin{pmatrix} M & 0 \\ 0 & M^{-1} \end{pmatrix}$ and that
$\rho_0(b')=\begin{pmatrix} b'_{11} & b'_{12} \\ b'_{21} & b'_{22} \end{pmatrix}$
satisfies $b'_{11}\ne M^{\pm 1}$.
Then there exists
a representation 
$\rho\in \text{\rm Hom}(\langle \hat a, \hat b\rangle,SL(2,\C))$ 
such that $\rho(\hat a)=\rho_0(\hat a)$ and $\rho(b')=\rho_0(b')$.
\end{lemma}

\begin{proof}
Set $\rho(\hat a)=\rho_0(\hat a)$.
We will find a $\rho$ such that $\rho(b')=\rho_0(b')$.
Set $\rho(\hat b)=\begin{pmatrix} b_{11} & b_{12} \\ b_{21} & b_{22} \end{pmatrix}\in SL(2,\C)$ and
let $f_{11}, f_{12}, f_{21}, f_{22}$ be the polynomial functions, in the variables $M$ and the $b_{ij}$'s, 
given by
\[
   \begin{pmatrix} f_{11} & f_{12} \\ f_{21} & f_{22} \end{pmatrix}
   =\rho(c)\rho(\hat b)\rho(c)^{-1},
\]
where $f_{11}f_{22}-f_{12}f_{21}=1$. 
We eliminate the variables $b_{22}$ and $b_{21}$ by substituting
$b_{21}=\frac{1}{b_{12}}(b_{11}b_{22}-1)$ 
and $b_{22}=M+\frac{1}{M}-b_{11}$,
where the second equation holds since $\hat a$ and $b'$ are conjugate.
The remaining variables are $M$,  $b_{11}$, and $b_{12}$.

We first prove that $f_{11}$ depends on the variables $b_{11}$ and $b_{12}$.
Assume it does not, i.e., $f_{11}$ is constant for each, fixed, choice of $M$.
Setting $\rho(\hat b)=\begin{pmatrix} M & 0 \\ 0 & M^{-1} \end{pmatrix}$ 
(resp.\ $\rho(\hat b)=\begin{pmatrix} M^{-1} & 0 \\ 0 & M \end{pmatrix}$)
we have
$\begin{pmatrix} f_{11} & f_{12} \\ f_{21} & f_{22} \end{pmatrix}
=\begin{pmatrix} M & 0 \\ 0 & M^{-1} \end{pmatrix}$
(resp.\
$\begin{pmatrix} f_{11} & f_{12} \\ f_{21} & f_{22} \end{pmatrix}
=\begin{pmatrix} M^{-1} & 0 \\ 0 & M \end{pmatrix}$).
Therefore we have $M=M^{-1}$, i.e., $M=\pm 1$ since $f_{11}$ is constant.
However, in the case $M=\pm 1$, since $\rho_0(\hat a)=\begin{pmatrix} \pm 1 & 0 \\ 0 & \pm 1 \end{pmatrix}$, 
the equality
\[
\begin{pmatrix} f_{11} & f_{12} \\ f_{21} & f_{22} \end{pmatrix}
=\rho(\hat b)=
\begin{pmatrix} b_{11} & b_{12} \\ b_{21} & b_{22} \end{pmatrix}
\]
is satisfied for any choice of the $b_{ij}$'s,
which contradicts the assumption that $f_{11}$ does not depend on $b_{11}$.

Now $f_{11}$ does depend on at least one of the variables $b_{11}$ and $b_{12}$,
so we solve the equation $f_{11}=b'_{11}$ in terms of one of these variables.
The inequality $f_{11}=b'_{11}\ne M^{\pm 1}$ implies $f_{12}\ne 0$ and $f_{21}\ne 0$,
otherwise we cannot have $f_{11}f_{22}-f_{12}f_{21}=1$.
For the same reason, we have $b'_{12}\ne 0$ and $b'_{21}\ne 0$.
The conjugation of $\rho$ by the matrix
\[
   P=\begin{pmatrix} \sqrt{b'_{12}/f_{12}} & 0 \\ 0 & \sqrt{f_{12}/b'_{12}}\end{pmatrix}
\]
satisfies
\[
   P\rho(\hat a)P^{-1}=\rho(\hat a)\quad\text{and}\quad
   P\begin{pmatrix} f_{11} & f_{12} \\ f_{21} & f_{22} \end{pmatrix}P^{-1}=
   \begin{pmatrix} b'_{11} & b'_{12} \\ b'_{21} & b'_{22} \end{pmatrix},
\]
where the bottom two equalities in the second matrix equation are automatically satisfied
by the equation $f_{11}+f_{22}=b'_{11}+b'_{22}$ and the fact that these matrices are in $SL(2,\C)$.
Hence we obtain the representation required.
\end{proof}

Let $f^+(M)$ be the rational function of one variable $M$ that appears as 
the top-right entry of $\rho(c)\rho(\hat b)\rho(c)^{-1}$
when
\[
   \rho(\hat a)=\begin{pmatrix} M & 0 \\ 0 & M^{-1} \end{pmatrix} \quad\text{and}\quad
\rho(\hat b)=\begin{pmatrix} M & 1 \\ 0 & M^{-1} \end{pmatrix}.
\]
Similarly, we define $f^-(M)$ to be the rational function of one variable $M$ that is 
the top-right entry of $\rho(c)\rho(\hat b)\rho(c)^{-1}$
when
\[
   \rho(\hat a)=\begin{pmatrix} M & 0 \\ 0 & M^{-1} \end{pmatrix} \quad\text{and}\quad
\rho(\hat b)=\begin{pmatrix} M^{-1} & 1 \\ 0 & M \end{pmatrix}.
\]

\begin{lemma}\label{lemma101c}
$f^+(M)$ and $f^-(M)$ are not constant.
\end{lemma}

\begin{proof}
We can set $\rho(c)=\begin{pmatrix} M^k & c_{12} \\ 0 & M^{-k} \end{pmatrix}$,
where $c_{12}$ is a rational function in one variable, $M$, whose 
denominator, if any, is a power of $M$,
and $k\in\mathbb Z$. Then 
\[
\begin{split}
   \rho(c)\rho(\hat b)\rho(c)^{-1}
   &=\begin{pmatrix} M^k & c_{12} \\ 0 & M^{-k} \end{pmatrix}
   \begin{pmatrix} M^{\pm 1} & 1 \\ 0 & M^{\mp 1} \end{pmatrix}
   \begin{pmatrix} M^{-k} & -c_{12} \\ 0 & M^k \end{pmatrix} \\
   &=\begin{pmatrix} M^{\pm 1} & (M^{k\mp 1}-M^{k\pm 1})c_{12}+M^{2k} \\ 0 & M^{\mp 1} \end{pmatrix},
\end{split}
\]
i.e.,
\[
   f^{\pm}(M)=(M^{k\mp 1}-M^{k\pm 1})c_{12}+M^{2k}.
\]
This cannot be constant since, even if $c_{12}$ has a denominator, it is only a power of $M$.
\end{proof}

\begin{lemma}\label{lemma101a}
Suppose that $\rho_0(\hat a)=\begin{pmatrix} M & 0 \\ 0 & M^{-1} \end{pmatrix}$ and 
$\rho_0(b')=\begin{pmatrix} M & b'_{12} \\ b'_{21} & M^{-1} \end{pmatrix}$ with
$b'_{12}b'_{21}=0$.
Suppose further that $f^+(M)\ne 0$.
Then there exists
a reducible representation 
$\rho\in \text{\rm Hom}(\langle \hat a, \hat b\rangle,SL(2,\C))$ 
such that $\rho(\hat a)=\rho_0(\hat a)$ and $\rho(b')=\rho_0(b')$.
\end{lemma}

\begin{proof}
Set $\rho(\hat a)=\rho_0(\hat a)$.
We will find a reducible representation $\rho$ such that $\rho(b')=\rho_0(b')$.
Consider the case where $b'_{21}=0$.
As above, we have $b'=c\hat b c^{-1}$.
Set $\rho(\hat b)=\begin{pmatrix} M & b_{12} \\ 0 & M^{-1} \end{pmatrix}$; then
the top-right entry of $\rho(c)\rho(\hat b)\rho(c)^{-1}$ becomes $f^+(M) b_{12}$.
Since $f^+(M)\ne 0$, $b_{12}=b'_{12}/f^+(M)$ gives the required reducible representation.
The proof for the case $b'_{12}=0$ is similar.
\end{proof}

\begin{lemma}\label{lemma101b}
Suppose that $\rho_0(\hat a)=\begin{pmatrix} M & 0 \\ 0 & M^{-1} \end{pmatrix}$ and 
$\rho_0(b')=\begin{pmatrix} M^{-1} & b'_{12} \\ b'_{21} & M \end{pmatrix}$ with
$b'_{12}b'_{21}=0$.
Suppose further that $f^-(M)\ne 0$. Then there exists
a reducible representation 
$\rho\in \text{\rm Hom}(\langle \hat a, \hat b\rangle,SL(2,\C))$ 
such that $\rho(\hat a)=\rho_0(\hat a)$ and $\rho(b')=\rho_0(b')$.
\end{lemma}

\begin{proof}
Similar to the proof of Lemma~\ref{lemma101a}.
\end{proof}

\noindent
{\it Proof of Theorem~\ref{thm02}.}\;\,
Let $\mathcal R(K)$ denote the representation variety $\text{Hom}(\pi_1(S^3\setminus K), SL(2,\C))$ 
of a knot $K$ in $S^3$.

Let $M$ and $M^{-1}$ be the eigenvalues of $\rho_0(\hat a)$.
Assume that $f^{\pm}(M)\ne 0$ and $M\ne \pm 1$.
Lemma~\ref{lemma101c} ensures that, except for a finite number of values, every $M\in\R$ satisfies these conditions.
Since $M\ne\pm 1$, $\rho_0(\hat a)$ is diagonalizable and hence we can 
set $\rho_0(\hat a)=\begin{pmatrix} M & 0 \\ 0 & M^{-1} \end{pmatrix}$ by conjugation.
Then 
by Lemma~\ref{lemma101z}, Lemma~\ref{lemma101a}, and Lemma~\ref{lemma101b},
for each representation $\rho_0\in\mathcal R(N(T))$,
there exists $\rho\in \text{\rm Hom}(\langle \hat a, \hat b\rangle ,SL(2,\C))$ such that
$\rho(\hat a)=\rho_0(\hat a)$ and $\rho(b')=\rho_0(b')$.
The quotient map $q:\pi_1(S^3\setminus N(S))\to \langle \hat a, \hat b\rangle$ induces a representation
$\rho\in \mathcal R(N(S))$ which satisfies
$\rho(\hat a)=\rho_0(\hat a)$ and $\rho(b')=\rho_0(b')$.
Let $D_{N(S+T)}$ be a knot diagram of $N(S+T)$ such that we can see the tangle decomposition
into $N(S)$ and $N(T)$ on that diagram. Fix a Wirtinger presentation of $\pi_1(S^3\setminus N(S+T))$
on $D_{N(S+T)}$.
Clearly, $\rho_0$ satisfies the relations of the Wirtinger presentation in the tangle $T$ and
$\rho$ also satisfies the relations in the tangle $S$.
Therefore these representations satisfy all the relations of the Wirtinger presentation, in other words,
we obtain an $SL(2,\C)$-representation of $\pi_1(S^3\setminus N(S+T))$.

Each irreducible component of $A^\circ_{N(T)}(L,M)=0$ corresponds to 
an irreducible component $Y$ of $\mathcal R(N(T))$ on which $M$ varies.
Since each representation $\rho_0\in Y$ corresponds to
a representation $\rho_1\in\mathcal R(N(S+T))$, except for a finite number of $M$ values,
there always exists a subvariety $Z$ in $\mathcal R(N(S+T))$ 
which corresponds to $Y$.

Let $Z_\Delta$ be the algebraic subset of $Z$ consisting of all $\rho_1 \in Z$
such that $\rho_1(\ell_1)$ and $\rho_1(m_1)$ are upper triangular,
where $(m_1, \ell_1)$ is the meridian-longitude pair of $N(S+T)$.
Let $\xi:Z_\Delta\to\C^2$ be the eigenvalue map $\rho_1\mapsto (L_1,M_1)$,
where $L_1$ and $M_1$ are the top-left entries of $\rho_1(\ell_1)$ and $\rho_1(m_1)$ respectively.
It is known by~\cite[Corollary~10.1]{CL} that $\dim \xi(Z_\Delta)\leq 1$. 
Since $M$ varies on $\xi(Z_\Delta)$, we have $\dim \xi(Z_\Delta)=1$.
This means that there exists a factor of the A-polynomial $A_{N(S+T)}(L,M)$
which vanishes at $(L,M)=(L_1,M_1)$.

In summary, for each generic point $(L_0,M_0)\in\{A^\circ_{N(T)}(L,M)=0\}$,
there is a representation $\rho_0\in\mathcal R(N(T))$ such that 
the top-left entries of $\rho_0(\ell_0)$ and $\rho_0(m_0)$ are $L_0$ and $M_0$ respectively,
where $(m_0, \ell_0)$ is the meridian-longitude pair of $N(T)$,
and there exists a representation $\rho_1\in\mathcal R(N(S+T))$ corresponding to $\rho_0$
such that the image $(L_1,M_1)$ satisfies $A_{N(S+T)}(L_1,M_1)=0$.
Thus if we have $\rho_0(m_0)=\rho_1(m_1)$ and $\rho_0(\ell_0)=\rho_1(\ell_1)$ then
$M_0=M_1$ and $L_0=L_1$, and hence we have $A_{N(S+T)}(L_0,M_0)=0$.
This means that the factor $A^\circ_{N(T)}(L,M)$ appears in $A_{N(S+T)}(L,M)$.
Since $m_0=m_1$ from the construction, we have $\rho_0(m_0)=\rho_1(m_1)$.
Hence, it is enough to show that $\rho_0(\ell_0)=\rho_1(\ell_1)$.

Let $\Sigma$ be the Seifert surface of $N(S+T)$ described on the diagram $D_{N(S+T)}$
by using Seifert's algorithm. The boundary of $\Sigma$ determines $\ell_1$.
Using the Wirtinger presentation of $\pi_1(S^3\setminus N(S+T))$ on $D_{N(S+T)}$, 
the longitude $\ell_1$ in $\pi_1(S^3\setminus N(S+T))$ is represented as
a product of words of the generators in the Wirtinger presentation by reading the words
along the boundary of $\Sigma$. This word presentation of $\ell_1$ has the form
\[
   \ell_1=\ell_{T,1} \ell_{S,1}\ell_{T,2}\ell_{S,2},
\]
where, for $i=1, 2$, $\ell_{T,i}$ is a product of generators in the tangle $T$
and $\ell_{S,i}$ is a product of generators in the tangle $S$.
Since each $\ell_{S,i}$ represents one of the boundary components of a Seifert surface of
the split link $N(S)$
and the representation $\rho_1$ is defined via the quotient map $q:\pi_1(S^3\setminus N(S))\to F_2$,
$\rho_1(\ell_{S,i})$ is the identity matrix.
Therefore we have $\rho_1(\ell_1)=\rho_1(\ell_{T,1})\rho_1(\ell_{T,2})=\rho_0(\ell_0)$.
\qed

\section{Examples and Questions}

\subsection{\RTR examples of 10 or fewer crossings}

\begin{defn}
A knot $K$ in $S^3$ is said to be an \RTR {\it knot} if it satisfies the following:
\begin{itemize}
\item[(1)] $K$ is of the form $N(R + T + \reflectbox{R})$, where
$R$ is rational, $\reflectbox{R}$ is the mirror reflection of $R$,
and $T$ is some tangle.
\item[(2)] $K$ is not isotopic to $N(T)$.
\end{itemize}
\end{defn}

The second condition is added to exclude trivialities,
for example the case where $R$ consists of two horizontal arcs.
Since $N(R + \reflectbox{R})$ is always a trivial link of two components,
$N(R + T + \reflectbox{R})$ satisfies the conditions of Theorem~\ref{thm02}
with $S=R+\reflectbox{R}$.

Here are two simple families of \RTR knots:
\begin{itemize}
\item The $2$-bridge knots of the form $[a_1,a_2, a_3,\cdots, a_k, \cdots, a_{2n-1}]$
with $a_i=-a_{2n-i}$ for $i=1,\cdots,n-1$ and $a_n$ odd.
\item Three-tangle Montesinos knots of the form $(p/q, r/s, -p/q)$.
\end{itemize}
Note that the infinite collection of pairs of $2$-bridge knots of Theorem~\ref{thm01} 
and Corollary~\ref{cor04} are included in the first of these families.

In the following, we represent the rational tangle corresponding to the rational number
$p/q$ by $R(p/q)$.
For example, the Montesinos knot of the form $(p/q, r/s, -p/q)$ is represented
as $N(R(p/q)+R(r/s)+R(-p/q))$.

Table~1 lists the \RTR\  knots of 10 or fewer crossings of which we know.
In the table, $T_0$ is the tangle obtained as the $+\pi/2$-rotation of
the tangle sum $R(-1/1)+R(1/3)+R(1/3)$ and
$T_1$ is obtained as the $+\pi/2$-rotation of 
the tangle sum $R(1/3)+R(-1/3)$.
We use $3_1^\text{mir}$ to denote the mirror image of $3_1$ and use
$\#$ for the connected sum of two knots.
In the table, we include information of epimorphisms among the knot groups and 
Alexander polynomials for convenience.  The epimorphism data is from~\cite{KS}.
In the column ``Alex. poly.,'' we represent a knot's Alexander polynomial 
by enclosing the knot's symbol in parenthesis.

\begin{table}
\caption{Factorizations of \RTR knots}\label{table1}
\begin{center}
\begin{tabular}{|c|c|c|c|c|c|}
\hline
    & \RTR & type & A-poly. fac. & epi. & Alex. poly. \\ \hline 
$8_{10}$   & $1/3, 3/2, -1/3$     & A & $3_1$ & $8_{10}\to 3_1$   & $(3_1)^3$ \\ \hline 
$8_{11}$   & $[2,-2,3,2,-2]$      & B & $3_1$ & No			      & $(3_1)(6_1)$ \\ \hline 
$9_{24}$   & $1/3, 5/2, -1/3$     & A & $4_1$ & $9_{24}\to 3_1$   & $(3_1)^2(4_1)$ \\ \hline 
$9_{37}$   & $1/3, 5/3, -1/3$     & B & $4_1$ & $9_{37}\to 4_1$   & $(4_1)(6_1)$ \\ \hline 
$10_{21}$  & $[2,-2,5,2,-2]$      & B & $5_1$ & No			      & $(5_1)(6_1)$ \\ \hline 
$10_{40}$  & $[2,2,3,-2,-2]$      & B & $3_1$ & $10_{40}\to 3_1$  & $(3_1)(8_8)$ \\ \hline 
$10_{59}$  & $2/5, 3/2, -2/5$     & A & $3_1$ & $10_{59}\to 4_1$  & $(3_1)(4_1)^2$ \\ \hline 
$10_{62}$  & $1/3, 5/4, -1/3$     & A & $5_1$ & $10_{62}\to 3_1$  & $(3_1)^2(5_1)$ \\ \hline 
$10_{65}$  & $1/3, 7/4, -1/3$     & A & $5_2$ & $10_{65}\to 3_1$  & $(3_1)^2(5_2)$ \\ \hline 
$10_{67}$  & $1/3, 7/5, -1/3$     & B & $5_2$ & No			      & $(5_2)(6_1)$ \\ \hline 
$10_{74}$  & $1/3, 7/3, -1/3$     & B & $5_2$ & $10_{74}\to 5_2$  & $(5_2)(6_1)$ \\ \hline 
$10_{77}$  & $1/3, 7/2, -1/3$     & A & $5_2$ & $10_{77}\to 3_1$  & $(3_1)^2(5_2)$ \\ \hline 
$10_{98}$  & $1/3,\; T_0,\, -1/3$ & B & $3_1\#3_1$ & $10_{98}\to 3_1$  & $(3_1)^2(6_1)$ \\ \hline 
$10_{99}$  & $1/3,\; T_1,\, -1/3$ & A & $3_1\#3_1^\text{mir}$ & $10_{99}\to 3_1$  & $(3_1)^4$ \\ \hline 
$10_{143}$ & $1/3, 3/4, -1/3$     & A & $3_1$ & $10_{143}\to 3_1$ & $(3_1)^3$ \\ \hline 
$10_{147}$ & $1/3, 3/5, -1/3$     & B & $3_1$ & No				  & $(3_1)(6_1)$ \\ \hline 
\end{tabular}
\end{center}
\end{table}

There are two types of \RTR\  knots depending on how the 
strands enter and leave
the tangle $T$. We say that the \RTR\  knot $N(R + T + \reflectbox{R})$ is {\it of type A}
if the tangle $T$ is a marked tangle. Otherwise we say it is {\it of type B}.

\begin{lemma}\label{lemma101}
Let $K=N(R + T + \reflectbox{\rm R})$ be an {\rm \RTR}knot with $R=R(p/q)$ and $q>0$. Then
\begin{itemize}
\item[(i)] $q>1$.
\item[(ii)] If $K$ is of type A then
$\Delta_{K}(t)=\Delta_{N(T)}(t)\Delta_{D(R)}(t)^2$.
\item[(iii)]
If $K$ is of type B then 
$\Delta_{K}(t)=\Delta_{N(T)}(t)\Delta_{N(\text{\rm R+R(1/1)+\reflectbox{\rm R}})}(t)$.
\item[(iv)]
The knot determinant of $K$ is divisible by $q^2$.
\end{itemize}
\end{lemma}

\begin{proof}
If $q=1$ then we have $N(R + T + \reflectbox{\rm R})=N(T)$.
Such a knot is not \RTR by definition. 
Thus we have assertion~(i).
Assertion~(ii) follows from equation~\eqref{alexfac} and the equations
\[
\Delta_{D(S)}(t)=\Delta_{D(R)\# D(\text{\reflectbox{R}})}(t)=\Delta_{D(R)}(t)^2.
\]
Next we prove assertion~(iii).
Since $K$ is of type B, we need to modify the diagram of $N(R + T + \reflectbox{\rm R})$
as shown in Figure~\ref{fig4} such that it becomes the sum of marked tangles.
We denote the marked tangle obtained from $T$ by $T'$ and the complementary tangle of $T'$ by $S'$.
From the figure, we can see that $D(S')=N(R + R(1/1) + \reflectbox{\rm R})$.
Thus assertion~(iii) follows from equation~\eqref{alexfac}.
\begin{figure}[htbp]
   \centerline{\input{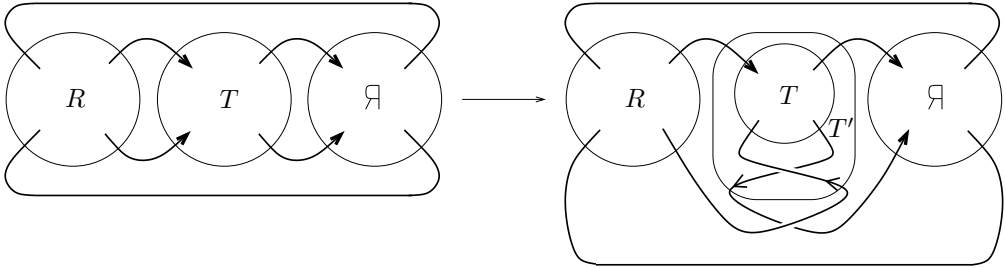}}
\caption{Changing an  \RTR knot of type B into the sum of two marked tangles.\label{fig4}}
\end{figure}

Finally, we check the last assertion.
It is known that the knot determinant of a knot is equal to the absolute value of its Alexander polynomial
evaluated at $t=-1$ (see for instance~\cite[Proposition~6.1.5]{murasugi}). 
We also know that the knot determinant of $D(R(p/q))$ is $q$.
Thus, if $K$ is of type A then assertion~(iv) follows immediately from
the factorization in~(ii).
Suppose $K$ is of type B.
Then, from
\[
N(R(p/q)+R(1/1)+R(-p/q))=N(R(p/q)+R((q-p)/q)),
\]
the knot determinant of $D(S')=N(R + R(1/1) + \reflectbox{\rm R})$
is calculated as $|pq+(q-p)q|=q^2$, see~\cite{ernst-sumners} and also~\cite[Theorem~9.3.5]{murasugi}.
Thus, from the factorization in~(iii), we again have assertion~(iv).
\end{proof}

\begin{prop}
Let $K$ be a prime knot of $10$ or fewer crossings.
Suppose that $K$ is not $8_{18}$, $9_{40}$, $10_{82}$, $10_{87}$, or $10_{103}$.
Then $K$ is {\rm \RTR}with $N(T)$ a non-trivial knot of $10$ or fewer crossings
if and only if it is in Table~1.
\end{prop}

We were unable to determine if the remaining five knots are \RTR or not.
In the proof below, we show further that if $8_{18}$, $9_{40}$, or $10_{82}$
is \RTR then it is of type A and if $10_{103}$ is \RTR then it is of type B. 

\begin{proof}
We first consider the case of type A.
Set $R=R(p/q)$ with $q>0$. By Lemma~\ref{lemma101}~(i), we have $q>1$.
We check if the Alexander polynomial of a knot, up to $10$ crossings,
has a factorization of the form in Lemma~\ref{lemma101}~(ii).
Since the knot determinant of $D(R)$ is equal to $\Delta_{D(R)}(-1)$,
$q>1$ implies that the polynomial $\Delta_{D(R)}(t)^2$ in Lemma~\ref{lemma101}~(ii) is non-trivial.
Now we check if the Alexander polynomial of a knot has such 
a non-trivial, multiple factor corresponding to a knot up to $10$ crossings. 
The candidate knots $K$ are
\[
   8_{10}, 8_{18}, 8_{20}, 9_{24}, 9_{40}, 10_{59}, 10_{62}, 10_{65}, 10_{77}, 10_{82}, 10_{87}, 10_{98},
   10_{99}, 10_{123}, 10_{137}, 10_{140}, 10_{143}.
\]
If $N(T)$ is one of $8_{20}$, $10_{123}$, $10_{137}$, and $10_{140}$ then we have $\Delta_{N(T)}(t)=1$,
i.e., $N(T)$ is trivial since it is assumed to be of $10$ or fewer crossings.
Such a case is excluded by assumption.
Since $8_{10}$, $9_{24}$, $10_{59}$, $10_{62}$, $10_{65}$, $10_{77}$, $10_{99}$, and $10_{143}$
are in Table~1, 
the remaining knots are only $8_{18}$, $9_{40}$, $10_{82}$, $10_{87}$, and $10_{98}$.
 
The Alexander polynomial of $10_{98}$ is $\Delta_{10_{98}}(t)=(2t-1)(t-2)(t^2-t+1)^2$
and hence we know that $N(T)$ is either $6_1$ or $9_{46}$.
In either case, $g_4(N(T))=0$, where $g_4(\cdot)$ represents the $4$-genus of a knot.
On the other hand, $g_4(10_{98})=2$.
Now, let $(F, \partial F)\subset (B^4, \partial B^4)$ 
be an orientable surface in the $4$-ball $B^4$ with $\partial F=N(T)$
such that the genus of $F$ is $0$, where $\partial B^4$ is the $3$-sphere bounding $B^4$.
Since $N(R + \reflectbox{\rm R})$ bounds two disjoint disks $D_1\sqcup D_2$ 
embedded in $\partial B^4$,
by gluing $F\subset B^4$ and $D_1\sqcup D_2\subset \bd B^4$ suitably,
we can obtain an orientable surface in $B^4$ bounded by $N(R + T+ \reflectbox{\rm R})=10_{98}$
and of genus $0$. However we have $g_4(10_{98})=2$ and this is a contradiction.
The remaining four knots are excluded by hypothesis.

Next we consider \RTR knots of type B. 
It is shown in the proof of Lemma~\ref{lemma101}~(iv) that 
the knot determinant of $D(S)$ is $q^2>1$.
Hence $D(R)$ is a $2$-bridge knot with denominator $q>1$.
Moreover, since $N(T)$ is assumed to be 
a non-trivial knot of $10$ or fewer crossings, $\Delta_{N(T)}(t)$ is non-trivial.
Hence we know that the factorization of $\Delta_K(t)$ in Lemma~\ref{lemma101}~(iii) is non-trivial.

Now we make a list of knots $K$, up to $10$ crossings that satisfy 
\begin{itemize}
\item the Alexander polynomial of $K$ factors into two non-trivial 
Alexander polynomials, and
\item the knot determinant of $K$ is divisible by $q^2$ for some integer $q>1$.
\end{itemize}
The following knots satisfy these conditions:
\[
8_{10}, 8_{11}, 8_{18}, 8_{20}, 9_1, 9_6, 9_{23},9_{24}, 9_{37}, 9_{40}, 
10_{21}, 10_{40}, 10_{59}, 10_{62}, 10_{65}, 10_{66}, 10_{67}, 
\]
\[
10_{74}, 10_{77}, 10_{82}, 10_{87}, 10_{98}, 10_{99}, 10_{103}, 10_{106},
10_{123}, 10_{137}, 10_{140}, 10_{143}, 10_{147}.
\]
Among them, the knots $8_{11}$, $9_{37}$, $10_{21}$, $10_{40}$, $10_{67}$, $10_{74}$, $10_{98}$,
and $10_{147}$ are in Table~1.

If $K$ is one of $8_{10}, 8_{18}, 8_{20}, 9_1, 9_6, 9_{23}, 9_{24}, 10_{62}, 10_{65}, 10_{77},
10_{82}, 10_{140}$, and $10_{143}$,
we have $q^2=9$, i.e., $q=3$.
The only knot of the form $N(R + 1 + \reflectbox{\rm R})$ with $q=3$ is $6_1$.
Since $\Delta_{6_1}(t)=(2t-1)(t-2)$, we have $(2t-1)(t-2)\mid \Delta_K(t)$.
However, none of the above knots satisfy this property. Hence they are not \RTR knots.

If $K$ is one of $9_{40}, 10_{59}, 10_{66}, 10_{103}, 10_{106}$, and $10_{137}$
we have $q^2=25$, i.e., $q=5$.
The only knots of the form $N(R + 1 + \reflectbox{\rm R})$ with $q=5$ are $8_8$ and $10_3$.
Since $\Delta_{8_8}(t)=(2t^2-2t+1)(t^2-2t+2)$ and
$\Delta_{10_3}(t)=(3t-2)(2t-3)$ one of theses is a factor of
$\Delta_K(t)$.
However, none of the above knots other than $10_{103}$ 
satisfy this property. Hence they are not \RTR knots.
As for $10_{103}$, it is excluded by hypothesis.

This leaves only $10_{87}$, $10_{99}$, and $10_{123}$ of which 
$10_{87}$ is excluded by hypothesis. If $K$ is $10_{99}$ then $q=3$ or $9$.
The case $q=3$ can not happen since $(2t-1)(t-2)$ is not a factor of $\Delta_{10_{99}}(t)$ as before.
If $q=9$ then, since the knot determinant of $10_{99}$ is $81$,
the factorization in Lemma~\ref{lemma101}~(iii) implies $\Delta_{N(T)}(t)=1$, i.e.,
$N(T)$ is trivial. Such a case is excluded by assumption.
If $K$ is $10_{123}$ then, since the knot determinant of $10_{123}$ is $121$,
we have $q=11$ and $\Delta_{N(T)}(t)=1$, i.e.,
$N(T)$ is trivial. This is again excluded. This completes the proof.
\end{proof}

\begin{rmk}
There is no direct relationship between the \RTR construction
and the list of epimorphisms in~\cite{KS}.
First of all, we can see from Table~1 that the following $9$ knots
\[
   8_{11}, 9_{24}, 10_{21}, 10_{59}, 10_{62}, 10_{65}, 10_{67}, 10_{77}, 10_{147}
\]
have the factorization of the A-polynomials but have no epimorphisms
to the corresponding knot groups.

Even for the other knots in Table~1, we believe that there is no relationship for the following reason:
In~\cite{KS2} it is written that there is an epimorphism 
$\pi_1(S^3\setminus 8_{10})\to \pi_1(S^3\setminus 3_1)$ which maps 
the longitude of $8_{10}$ to $1\in\pi_1(S^3\setminus 3_1)$, see~Figure~\ref{fig3}, 
while Theorem~\ref{thm02}
shows that the longitude of $8_{10}$ corresponds to that of $3_1$ in our construction.
In this example, the epimorphism is given by the tangle $R$ and
the factorization of the A-polynomial is given by the tangle $T$.
In general, for any \RTR knot of type A,
there is an epimorphism from $\pi_1(S^3\setminus N(R+T+\reflectbox{R}))$ 
to $\pi_1(S^3\setminus D(R))$ 
such that the image of the longitude of this \RTR knot is $1\in \pi_1(S^3\setminus D(R))$;
however, the longitude of $N(R+T+\reflectbox{R})$ corresponds to that of $D(R)$ 
when we compare their A-polynomials. This shows that the type A examples do not 
correspond to the epimorphisms. 
We remark that there may exist other epimorphisms from
$\pi_1(S^3\setminus N(R+T+\reflectbox{R}))$ to $\pi_1(S^3\setminus D(R))$ preserving peripheral structure.
This is why we cannot exclude the possibility
that there is a relationship between 
the factorization of A-polynomials and epimorphisms for these examples.
\end{rmk}

\begin{figure}[htbp]
   \centerline{\input{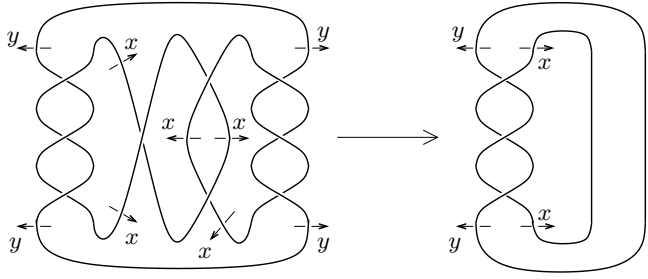}}
\caption{An epimorphism $\pi_1(S^3\setminus 8_{10})\to \pi_1(S^3\setminus 3_1)$.
Construct a Seifert surface and observe that the longitude of $8_{10}$ vanishes
in $\pi_1(S^3\setminus 3_1)$.\label{fig3}}
\end{figure}



\subsection{$r$--curves}

The idea of an $r$--curve in the character variety of a knot was introduced by Boyer and Zhang~\cite{BZ2}.
A natural question to ask is, which $r$ can occur? Taking advantage of mirror reflections, we can assume $r \geq 0$. If we take $N(T)$ to be a 
$(p,q)$-torus knot, then any $N(S+T)$ with $N(S)$ a split link
will have a $pq$-curve in its character variety.
It is also known that the $(-2,3,-3)$-pretzel knot admits a $0$-curve, see~\cite{M}, and 
this argument can be generalized to show that $(-2,p,-p)$-pretzel knots ($p$ odd) all have $0$-curves.
Together, this covers all integers except for prime powers. 
Calculations by Culler~\cite{KnotInfo} do provide a few examples of such $r$-curves:
$8_{21}$ has a $2$--curve,  $9_{37}$ has a $4$--curve, $10_{143}$   has an $8$--curve, and
$10_{152}$ has an $11$--curve. Still, we can ask, are there other examples, and more generally:

\begin{ques} Are there families of knots with $r$--curves such that $r$ is a prime power?
\end{ques}

The following related question was suggested by Thang~L\^{e}:

\begin{ques}
Is there an example of a $2$--bridge knot with a $0$--curve?
\end{ques}

According to a calculation by Culler~\cite{KnotInfo},  the knot $9_{38}$ has a $\frac10$--curve. Aside from this, there are no known examples of 
a knot with non-integral $r$--curve. It is known that a small knot admits no such $r$-curve, see \cite{BZ2}. This leads to our final question.

\begin{ques} Other than $9_{38}$, are there examples of knots with $r$--curves such that $r$ is not integral?
\end{ques}

\appendix
{
\section{On the roots of Alexander polynomials}

In this appendix we prove Theorem~\ref{thm01a}, stated below,
using reducible representations in $SL(2,\C)$. 
The assertion is a weaker version of the factorization of Alexander polynomials 
explained in the introduction and the proof is somehow analogous to the proof of Theorem~\ref{thm02}.

\begin{thm}\label{thm01a}
Any root of $\Delta_{N(T)}(t)=0$ is also a root of $\Delta_{N(S+T)}(t)=0$.
\end{thm}

Set $\hat a, b'$ to be elements in 
$\langle a, b\rangle=F_2$ 
corresponding to the two strands of the numerator closure of the tangle $T$ as before.
We may assume that $a$ and $\hat a$ (resp.\ $b$ and $b'$) are conjugate.
Furthermore, by replacing $b$ and $b'$ by $b^{-1}$ and $(b')^{-1}$ if necessary,
we assume that $\hat a$ and $b'$ are conjugate in $\pi_1(S^3\setminus N(T))$.
By Lemma~\ref{lemma100}, there exists an element $\hat b$ conjugate to $b$  
such that $\langle \hat a, \hat b\rangle=\langle a, b\rangle$.
Since $b'$ is conjugate to $b$, there exists $c\in \langle \hat a, \hat b\rangle$ such that
$b'=c\hat bc^{-1}$.
Let $\mathcal R(K)$ denote the representation variety
$\text{Hom}(\pi_1(S^3\setminus K),SL(2,\C))$ of a knot $K$.

It is shown in Lemma~\ref{lemma101a} that if $f^+(M)\ne 0$ then a reducible representation
$\rho_0\in \mathcal R(N(T))$ extends to a reducible representation of $N(S)$.
In the case $f^+(M)=0$, we do not know if the same extension is possible.
Nevertheless, we can prove Theorem~\ref{thm01a} by showing the existence of 
a representation of $N(S+T)$ for $M\in \R$ with $f^+(M)=0$ directly.

\begin{lemma}\label{lemma103}
Suppose that $M$ is a root of $f^+(M)=0$. Then there exists
a reducible representation $\rho\in \mathcal R(N(S+T))$ 
with non-abelian image
such that $\rho(\hat a)=\begin{pmatrix} M & 0 \\ 0 & M^{-1} \end{pmatrix}$.
\end{lemma}

\begin{proof}
Suppose that $M$ is a root of $f^+(M)=0$.
Set $\rho(\hat a)=\begin{pmatrix} M & 0 \\ 0 & M^{-1} \end{pmatrix}$ and
$\rho(\hat b)=\begin{pmatrix} M & b_{12} \\ 0 & M^{-1} \end{pmatrix}$; then
$\rho$ satisfies
\[
   \rho(\hat a)=\rho(b')=\begin{pmatrix} M & 0 \\ 0 & M^{-1} \end{pmatrix}.
\]
If we set the representation of each generator of the Wirtinger presentation of $\pi_1(S^3\setminus N(S+T))$
in the tangle $T$ to be
$\begin{pmatrix} M & 0 \\ 0 & M^{-1} \end{pmatrix}$, then 
the relations in the Wirtinger presentation on $T$ are automatically satisfied.
The quotient map $q:\pi_1(S^3\setminus N(S))\to \langle \hat a, \hat b\rangle$ 
induces a reducible representation $\rho\in \mathcal R(N(S))$ that satisfies
$\rho(\hat a)=\rho_0(\hat a)$ and $\rho(b')=\rho_0(b')$.
Such a $\rho$ is an element in $\mathcal R(N(S+T))$ by  construction.
Choosing $b_{12}\ne 0$, we have a reducible representation with non-abelian image.
\end{proof}

\noindent
{\it Proof of Theorem~\ref{thm01a}.\;\,}
Let $M^2$ be a root of $\Delta_{N(T)}(t)=0$.
Then there exists a reducible representation $\rho_0\in\mathcal R(N(T))$ which has
non-abelian image and sends the meridian to an element with eigenvalues $M$ and $M^{-1}$ 
by~\cite[Section~6.1]{CCGLS}.
Note that $M^2\ne 1$ since $\Delta_{N(T)}(1)=1$ and that $M=0$ is excluded because
the Alexander polynomial is defined up to multiplication by powers of $t$.
Assume $\rho_0(\hat a)=\begin{pmatrix} M & 0 \\ 0 & M^{-1} \end{pmatrix}$ by conjugation of $\rho_0$.

If $f^+(M)\ne 0$ then Lemma~\ref{lemma101a} 
ensures that there exists a reducible representation $\rho$ of 
$\langle \hat a, \hat b\rangle$ 
that satisfies $\rho(\hat a)=\rho_0(\hat a)$ and $\rho(b')=\rho_0(b')$.
The quotient map $q:\pi_1(S^3\setminus N(S))\to \langle \hat a, \hat b\rangle$ 
induces a reducible representation $\rho\in\mathcal R(N(S))$ that satisfies
$\rho(\hat a)=\rho_0(\hat a)$ and $\rho(b')=\rho_0(b')$.
Such a $\rho$ is an element in $\mathcal R(N(S+T))$ by construction.
This $\rho$ has non-abelian image because this is already true of its restriction to $\pi_1(S^3\setminus N(T))$.
Moreover, it obviously sends the meridian to an element with eigenvalues $M$ and $M^{-1}$.
Hence, by~\cite[Section~6.1]{CCGLS}, $M^2$ is a root of $\Delta_{N(S+T)}(t)=0$.

If $f^+(M)=0$ then Lemma~\ref{lemma103} ensures the existence of such a $\rho$.
Hence $M^2$ is a root of $\Delta_{N(S+T)}(t)=0$ in any case.
\qed
}

\end{document}